\documentclass{amsart}
\usepackage{graphicx}
\input xy
\xyoption{all}
\newenvironment{theorem}[2][Theorem]{\begin{trivlist}
\item[\hskip \labelsep {\bfseries #1}\hskip \labelsep {\bfseries #2}]}{\end{trivlist}}

\newtheorem{thm}{Theorem}

\newtheorem{lem}[thm]{Lemma}
\newtheorem{prop}[thm]{Proposition}
\theoremstyle{definition}
\newtheorem{defn}[thm]{Definition}
\theoremstyle{remark}
\newtheorem{rem}[thm]{Remark}
\theoremstyle{remark}
\newtheorem{rems}[thm]{Remarks}
\theoremstyle{definition}
\newtheorem{ex}[thm]{Examples}
\numberwithin{equation}{section}

\begin{document}

\title[Some homological criteria]{Some homological criteria for regular, complete intersection and Gorenstein rings}%
\author{Javier Majadas}%
\address{Departamento de \'Algebra, Facultad de Matem\'aticas, Universidad de Santiago de Compostela, E15782 Santiago de Compostela, Spain}%
\email{j.majadas@usc.es}%

\keywords{regular local ring, finite flat dimension}%
\thanks{2010 {\em Mathematics Subject Classification.} 13H05, 13H10, 13D03, 13D05}

\begin{abstract}
  Regularity, complete intersection and Gorenstein properties of a local ring can be characterized by homological conditions on the canonical homomorphism into its residue field (Serre, Avramov, Auslander). It is also known that in positive characteristic, the Frobenius endomorphism can also be used for these characterizations (Kunz, ...), and more generally any contracting endomorphism. We introduce here a class of local homomorphisms, in some sense larger than all above, for which these characterizations still hold, providing an unified treatment for this class of homomorphisms.
\end{abstract}
\maketitle

For a module of finite type $M$ over a commutative noetherian local ring $(A,\mathfrak{m},k)$, we consider the projective dimension $pd_A(M)$, the complete intersection dimension \cite{AGP} CI-dim$_A(M)$, and the Gorenstein dimension \cite{AB} G-dim$_A(M)$. We have
\begin{theorem}{} (\cite{Se}, \cite{AGP}, \cite{AB}) \\*
\emph{(i) $A$ is regular $\Leftrightarrow$ pd$_A(k)<\infty$\\*
(ii) $A$ is complete intersection $\Leftrightarrow$ CI-dim$_A(k)<\infty$\\*
(iii) $A$ is Gorenstein $\Leftrightarrow$ G-dim$_A(k)<\infty$.\\*}
\end{theorem}

Subsequently, some results were obtained showing that if $A$ contains a field of positive characteristic, the above results are also valid if we replace the map $A \rightarrow k$ by the Frobenius endomorphism $\phi :A \rightarrow A$. We summarize them in:

\begin{theorem}{} (\cite{Ku}, \cite{Ro}, \cite{BM}, \cite{He}, \cite{TY}, \cite{IS}) \\*
\emph{Let $A$ be a noetherian local ring containing a field of characteristic $p>0$, let $\phi :A \rightarrow A, \phi(a)=a^p$ be the Frobenius homomorphism, and let $^{\phi}A$ be the ring $A$ considered as $A$-module via $\phi$. Then:\\*
(i) $A$ is regular $\Leftrightarrow$ fd$_A(^{\phi}A)<\infty$\\*
(ii) $A$ is complete intersection $\Leftrightarrow$ CI-dim$_A(^{\phi}A)<\infty$\\*
(iii) $A$ is Gorenstein $\Leftrightarrow$ G-dim$_A(^{\phi}A)<\infty$}
\end{theorem}
(here, $fd$ denotes flat dimension; also, CI-dim and G-dim must be modified since $^{\phi}A$ is not in general of finite type over $A$).\\

Some of these results were extended from the Frobenius endomorphism to the more general class of contracting endomorphisms (that is, endomorphisms $f:A \rightarrow A$ such that $f^i(\mathfrak{m}) \subset \mathfrak{m}^2$ for some $i>0$) in \cite {KL}, \cite {TY}, \cite {IS}, \cite {AIM}, \cite {Ra}, \cite {AHIY}.

In this paper we introduce a class of local homomorphisms (called $h_2$-vanishing), and for this class we prove similar criteria for regularity, complete intersection and Gorensteinness. The interest in working with this class is twofold:

First, our class contains both the canonical epimorphism into the residue field and (powers of) all contracting endomorphisms. In fact, our class is much larger. It does not only contains the canonical epimorphism $A \rightarrow k$ of a local ring into its residue field, but more generally any local homomorphism $A \rightarrow B$ factorizing through a regular ring (Example \ref{ex}.i), and even in the case of an endomorphism, the proof given in Example \ref{ex}.ii gives us an idea of how more is needed for an $h_2$-vanishing endomorphism to be contracting. Note for example that if $A$ has a contracting endomorphism, then necessarily $A$ contains a field (the subring fixed by $f$), while this is not needed for $h_2$-vanishing homomorphisms.

Second, our treatment for this class of homomorphisms is uniform, giving a single proof valid in particular at once for the residue field and for the Frobenius endomorphism. This fact is a step in allowing us to understand better why some (and what) classes of local homomorphisms are "test" for regularity, complete intersection and Gorensteinness. We obtain these results as easy consequences of strong theorems in Andr\'e-Quillen homology of commutative algebras. In particular, the following theorem by Avramov \cite{Av-a} will play a key role:

\begin{theorem}{}
\emph{Let $f:(A,\mathfrak{m},k)\rightarrow(B,\mathfrak{n},l)$ be a local homomorphism of noetherian local rings. Let ${a_1,...a_r}$ be a minimal set of generators of the maximal ideal $\mathfrak{m}$ of $A$, and ${f(a_1),...f(a_r),b_1,...,b_s}$ a set of generators of the ideal $\mathfrak{n}$. Assume that $fd_A(B)<\infty$. Then, the induced homomorphism between the first Koszul homology modules
$$ H_1(a_1,...a_r;A)\otimes _kl \rightarrow H_1(f(a_1),...f(a_r),b_1,...,b_s;B)$$
is injective.\\}
\\
\end{theorem}

For convenience of the reader, we will start by recalling here some facts of Andr\'e-Quillen homology that we are going to use through the paper. Associated to a homomorphism of (always commutative) rings $f:A \rightarrow B$ and to a $B$-module $M$ we have Andr\'e-Quillen homology $B$-modules $H_n(A,B,M)$ for all integers $n \geq 0$, which are functorial in all three variables. \\

1. If $B=A/I$, then $H_0(A,B,M)=0$, $H_1(A,B,M)=I/I^2 \otimes_BM$  \cite [4.60, 6.1]{An}. \\

2. (Base change) Let $A \rightarrow B$, $A \rightarrow C$ be ring homomorphisms such that $B$ or $C$ is flat as $A$-module, and let $M$ be a $B\otimes_AC$-module. Then $H_n(A,B,M)=H_n(C,B\otimes_AC,M)$ for all $n$ \cite [4.54]{An}. \\

3. Let $B$ be an $A$-algebra, $C$ a $B$-algebra and $M$ a flat $C$-module. Then $H_n(A,B,M)=H_n(A,B,C) \otimes_CM$ for all $n$ \cite [3.20]{An}. \\

4. (Jacobi-Zariski exact sequence) If $A \rightarrow B \rightarrow C$ are ring homomorphisms and $M$ is a $C$-module, we have a natural exact sequence \cite [5.1]{An}

\begin{align*}
... \rightarrow H_{n+1}(B,C,M)\to \\
H_n(A,B,M) \rightarrow H_n(A,C,M)\rightarrow H_n(B,C,M)\rightarrow \\
H_{n-1}(A,B,M) \rightarrow \hspace{5mm}
... \hspace{5mm} \rightarrow H_0(B,C,M)\rightarrow 0\\
\end{align*}
5. If $K\rightarrow L$ is a field extension and $M$ an $L$-module, we have $H_n(K,L,M)=0$ for all $n \geq 2$ \cite [7.4]{An}. So if $A\rightarrow K \rightarrow L$ are ring homomorphisms with $K$ and $L$ fields, from 4 we obtain $H_n(A,K,L)=H_n(A,L,L)$ for all $n \geq 2$, which, using 3, gives $H_n(A,K,K) \otimes_KL=H_n(A,L,L)$ for all $n \geq 2$. \\

6. If $I$ is an ideal of a noetherian local ring $(A,\mathfrak{m},k)$, then the following are equivalent: \\*
(i) $I$ is generated by a regular sequence \\*
(ii) $H_2(A,A/I,k)=0$ \\*
(iii) $H_n(A,A/I,M)=0$ for any $A/I$-module $M$ for  all $n \geq 2$ \cite [6.25]{An}. \\

In particular, a noetherian local ring $(A,\mathfrak{m},k)$ is regular if and only if $H_2(A,k,k)=0$. \\

7. If $(A,\mathfrak{m},k)$ is a noetherian local ring and $\hat{A}$ is its $\mathfrak{m}$-completion, then $H_n(A,k,k)=H_n(\hat{A},k,k)$ for all $n \geq 0$ \cite [10.18]{An}. \\

8. If $(A,\mathfrak{m},k)\rightarrow(B,\mathfrak{n},l)$ is a local homomorphism of noetherian local rings with $fd_A(B)<\infty$, then  the theorem of Avramov cited above says that the homomorphism $H_2(A,l,l) \rightarrow H_2(B,l,l)$ is injective \cite [15.12]{An} (see details in the proof of \cite [4.2.2]{MR}). In fact, Avramov shows much more. In Remark \ref{h4} we will use that $H_4(A,l,l) \rightarrow H_4(B,l,l)$ is also injective.\\

\begin{defn}\label{h2}
Let $f:(A,\mathfrak{m},k)\rightarrow(B,\mathfrak{n},l)$ be a local homomorphism of noetherian local rings. We say that $f$ has the \emph{$h_2$-vanishing property} if the homomorphism induced by $f$ \\
$$H_2(A,l,l) \rightarrow H_2(B,l,l)$$
vanishes.
\end{defn}

\begin{prop}\label{main} If $f:(A,\mathfrak{m},k)\rightarrow(B,\mathfrak{n},l)$ has the $h_2$-vanishing property and there exists a local homomorphism of noetherian local rings $B \rightarrow C$ such that $fd_A(C)<\infty$ (that is, $Tor^A_n(C,-)=0$ for all $n\gg 0$), then $A$ is a regular local ring.
\end{prop}

\begin{proof}
Since the composition $A \rightarrow B \rightarrow C$ has the $h_2$-vanishing property, we can assume B=C. Therefore, the zero map $H_2(A,l,l) \rightarrow H_2(B,l,l)$ is injective by Avramov's theorem, so $H_2(A,l,l)=0$ and then $A$ is regular.
\end{proof}

\begin{ex}\label{ex}
(i) If the homomorphism $f:A \rightarrow B$ factorizes into local homomorphisms $f:A \rightarrow R \rightarrow B$ where $R$ is a regular local ring, then $f$ has the $h_2$-vanishing property. This includes as a particular case the canonical epimorphism $A \rightarrow k$ of a local ring $(A,\mathfrak{m},k)$ into its residue field.\\

(ii) Let $f:(A,\mathfrak{m},k)\rightarrow (A,\mathfrak{m},k)$ be a contracting homomorphism \cite [\S12]{AIM}, that is, there exists some $i>0$ such that $f^i(\mathfrak{m})\subset \mathfrak{m}^2$. That means that for any $t>0$, $f^j(\mathfrak{m}) \subset \mathfrak{m}^t$ for some $j>0$. We will see that some power of $f$ has the $h_2$-vanishing property and how to apply Proposition \ref{main} to these endomorphisms.

Assume first that $A$ is complete. Since $f$ is contracting, the subring $F\subset A$ of elements fixed by $f$ is clearly a field \cite [5.9]{AHIY}. Let $F_0$ be a perfect subfield of $F$. Then $F_0 \rightarrow k$ is formally smooth and so there exists a section $h$ of the $F_0$-algebra homomorphism $A \rightarrow k$ \cite [2.2.3]{MR}. So $k_0 := Im(h) \subset A$ is a coefficient field for $A$, and there exists a surjective ring homomorphism $\pi :R:=k_0[[X_1,...,X_n]] \rightarrow A$ sending $X_1,...,X_n$ to a set of generators $x_1,...,x_n$ of the maximal ideal $\mathfrak{m}$ of $A$. Let $I$ be its kernel. For each $j$, choose a power series $P_j(x_1,...,x_n)$ in $x_1,...,x_n$ over $k_0$ representing $f^i(x_j) \in \mathfrak{m}^2$. We can choose $P_j$ of order $\geq 2$. Since $F_0$ is invariant by $f$, we have $F_0$-algebra homomorphisms

$$
\xymatrix{  k_0 \ar @{^{(}->}[d] & R \ar @{->>}[d]^{\pi} \\
A \ar[r]^{f^i}   & A }
$$

By \cite [0$_{IV}$,19.3.10]{EGA} there exists a $F_0$-algebra homomorphism $g_0:k_0 \rightarrow R$ making commutative the square. We define $g:R \rightarrow R$ by $g_{|k_0}:= g_0$, $g(X_j):=P_j(X_1,...,X_n)$. Since $g$ is a lifting of $f^i$, we have $g(I)\subset I$. Also, $g(\mathfrak{q})\subset \mathfrak{q}^2$, where $\mathfrak{q}=(X_1,...,X_n)$ is the maximal ideal of $R$.

Therefore, for any $t>0$ there exists some $s$ such that $g^s(\mathfrak{q})\subset \mathfrak{q}^t$, and then $g^s(I)\subset \mathfrak{q}^t\bigcap I$. By the Artin-Rees lemma, some $s$ verifies $g^s(I)\subset \mathfrak{q}I$. Consider the Jacobi-Zariski exact sequence associated to $R \rightarrow A \rightarrow k$
$$0=H_2(R,k,k) \rightarrow H_2(A,k,k) \rightarrow I/\mathfrak{q}I$$
The endomorphism induced by $f^{is}$ on $H_2(A,k,k)$ is the restriction of the one induced by $g^s$ on $I/\mathfrak{q}I$ which is zero. We have proved that if $f:A \rightarrow A$ is contracting, then some power $f^{is}$ of $f$ has the $h_2$-vanishing property when $A$ is complete. But since $H_2(A,k,k) = H_2(\hat{A},k,k)$, this is also valid when $A$ is not complete.\\

Note that the fact that we have had to replace $f$ by some power $f^{is}$, does not interfere with our result. More generally, let $f:A \rightarrow A$ be contracting and assume that there exists a noetherian local $A$-algebra $(C,\mathfrak{p},l)$ such that $fd_A(^{f^t}C)<\infty$ for \emph{some} $t$. By Avramov's theorem, the homomorphism $H_2(A,k,l) \rightarrow H_2(C,l,l)$ induced by $A\xrightarrow{f^t} A \rightarrow C$ is injective, and then so is the endomorphism $H_2(A,k,k) \rightarrow H_2(A,k,k)$ induced by $f^t$. Thus for any $r$, the endomorphism $H_2(A,k,k) \rightarrow H_2(A,k,k)$ induced by $f^{tr}$ is also injective. But choosing $r$ so that $f^{tr}$ has the $h_2$-vanishing property, we obtain $H_2(A,k,k)=0$ and then $A$ is regular. This case of a contracting endomorphism was proved in \cite [Proposition 2.6]{KL} and \cite [Remark 4.4]{TY} and includes the particular case of a finite $A$-algebra $C$ (localizing $C$ at a prime ideal contracting in $A$ to $\mathfrak{m}$) with $fd_A(^{f^t}C)<\infty$, which was proved with different methods in \cite {AHIY} in the more general case of a finite $A${\it-module} $C\neq0$ with $fd_A(^{f^t}C)<\infty$.\\

(iii) The good properties of Andr\'e-Quillen homology (for instance the ones stated above and its behaviour under tensor products \cite [5.21]{An}) allows us to obtain new examples of $h_2$-vanishing homomorphisms from others.
\end{ex}

\begin{rems}\label{hn}
(i) It can be proved that if $f:A \rightarrow A$ is contracting, then for any $n \geq 0$, there exists $s$ (depending on $n$) such that $f^{s}$ has the $h_n$-vanishing property, that is, it induces the zero map $H_n(A,k,k)\rightarrow H_n(A,k,k)$ \cite [Proposition 10]{Ma}. However, our \textit{ad hoc} proof given here for the case $n=2$, besides being simpler and shorter, gives a better idea of the relationship between contracting and $h_2$-vanishing.\\*
(ii) If $f:A \rightarrow A$ is the Frobenius endomorphism, then $f^{s}$ has the $h_n$-vanishing property for all $s > 0$, $n \geq 0$ \cite [Lemme 53]{An-b}.

\end{rems}

\hspace{25mm}

Avramov's theorem also allows us to obtain similar criteria for complete intersection, Gorenstein, and Cohen-Macaulay rings, provided we use the adequate definitions for homological dimensions in terms of ``deformations" (see \cite [\S8]{Av-c}) and flat dimension. We start by recalling the definition of upper complete intersection dimension introduced in \cite {T} (see also \cite {AGP}).

We say that a finite module $M\neq0$ over a noetherian local ring $A$ has finite upper complete intersection dimension and denote it by CI*-dim$_A(M)<\infty$ if there exists a flat local homomorphism of noetherian local rings $(A,\mathfrak{m},k) \rightarrow (A',\mathfrak{m}',k')$ such that $A'\otimes_Ak$ is a regular local ring, and a surjective homomorphism of noetherian local rings $Q \rightarrow A'$ with kernel generated by a regular sequence, such that $pd_Q(M\otimes_AA')<\infty$, where $pd$ denotes projective dimension.

If $f:(A,\mathfrak{m},k)\rightarrow(B,\mathfrak{n},l)$ is a local homomorphism of noetherian local rings, a Cohen factorization of $f$ is a factorization $A\xrightarrow{i} R \xrightarrow{p} B$ of $f$ where $R$ is a noetherian local ring, $i$ is a flat local homomorphism, $R\otimes_Ak$ is a regular local ring and $p$ is surjective. If $B$ is complete, a Cohen factorization always exists \cite {AFH}.

We say that a local homomorphism of noetherian local rings $f:(A,\mathfrak{m},k)\rightarrow(B,\mathfrak{n},l)$ is of finite upper complete intersection dimension (and denote it by CI*-dim$(f)<\infty$) if there exists a Cohen factorization $A\rightarrow R \rightarrow \hat{B}$, such that CI*-dim$_R(\hat{B})<\infty$.

\begin{lem}\label{lemma} {\rm (Essentially \cite [lemma 1.7]{Av-b})}
Let $(A,k)\rightarrow(R,l)\rightarrow(D,E)$ be local homomorphisms of noetherian local rings such that $R$ is a flat $A$-module and $R\otimes_Ak$ is regular. Then $H_n(A,D,E)=H_n(R,D,E)$ for all $n\geq 2$.
\end{lem}
\begin{proof}
By flat base change $H_n(A,R,E)=H_n(k,R\otimes_Ak,E)$, and by the Jacobi-Zariski exact sequence associated to $k \rightarrow R\otimes_Ak \rightarrow E$ we have $H_n(k,R\otimes_Ak,E)=H_{n+1}(R\otimes_Ak,E,E)=0$ for all $n \geq 2$. So the Jacobi-Zariski exact sequence
$$... \rightarrow H_n(A,R,E) \rightarrow H_n(A,D,E) \rightarrow H_n(R,D,E) \rightarrow H_{n-1}(A,R,E) \rightarrow ...$$
gives isomorphisms $H_n(A,D,E)=H_n(R,D,E)$ for all $n\geq 3$ and an exact sequence
$$0 \rightarrow  H_2(A,D,E) \rightarrow H_2(R,D,E) \rightarrow H_1(A,R,E) \xrightarrow{\alpha} H_1(A,D,E) \rightarrow ...$$

The injectivity of $\alpha$ follows from the commutative diagram with exact upper row
$$
\xymatrix{  0=H_2(R\otimes_Ak,E,E) \ar[r] & H_1(k,R\otimes_Ak,E) \ar[r] & H_1(k,E,E)  \\
&  H_1(A,R,E) \ar[u]^{\simeq} \ar[r]^{\alpha}  & H_2(A,D,E) \ar[u] }
$$
\end{proof}

\begin{prop}\label{ci}
If $f:(A,\mathfrak{m},k)\rightarrow(B,\mathfrak{n},l)$ has the $h_2$-vanishing property and there exists a local homomorphism of noetherian local rings $g:(B,\mathfrak{n},l) \rightarrow (C,\mathfrak{p},E)$ such that CI*-dim$(gf)<\infty$, then $A$ is a complete intersection ring.
\end{prop}

\begin{proof}
Let $A\rightarrow R \rightarrow \hat{C}$ be a Cohen factorization, $R \rightarrow R'$ a flat local homomorphism with regular closed fiber, $Q \rightarrow R'$ a surjective homomorphism of noetherian local rings with kernel generated by a regular sequence such that $pd_Q(\hat{C}\otimes_RR')<\infty$. We have a commutative triangle

$$
\xymatrix{ & H_2(R,E,E) \ar[rd]^ \gamma  \\
H_2(A,E,E) \ar[ru]^ \beta \ar[rr]^\alpha && H_2(\hat{C},E,E)}
$$
where $\alpha = 0$ since $f$ has the $h_2$-vanishing property, and $\beta$ is surjective by Lemma \ref{lemma}. Then $\gamma =0$ and so the homomorphism $H_2(R,E'E') \rightarrow H_2(\hat{C},E',E')$ also vanishes, where $E'$ is the residue field of $R'$ and $\hat{C}\otimes_RR'$. We have a commutative diagram

$$
\xymatrix{  H_2(R,E',E') \ar[r]^0 \ar[d]^{\lambda} & H_2(\hat{C},E',E') \ar[d]\\
H_2(R',E',E') \ar[r]^{\mu}  & H_2(\hat{C}\otimes_RR',E',E') }
$$
where $\lambda$ is an isomorphism by Lemma \ref{lemma}. We see that $\mu =0$, that is $R' \rightarrow \hat{C}\otimes_RR'$ has the $h_2$-vanishing property. Composing with $Q \rightarrow R'$, we deduce that $Q \rightarrow \hat{C}\otimes_RR'$  has the $h_2$-vanishing property. By Proposition \ref{main}, $Q$ is a regular local ring, and then $R'$ is a complete intersection ring. By flat descent (\cite {Av-d} or \cite {Av-a}), $A$ is a complete intersection ring.
\end{proof}

Following in part \cite {Ve} we define the upper Gorenstein dimension as follows. If $M\neq0$ is a finite module over a noetherian local ring $A$, we say that $M$ has finite upper Gorenstein dimension if there exists a flat local homomorphism of noetherian local rings $A \rightarrow A'$ with regular closed fiber and a surjective homomorphism of noetherian local rings $Q \rightarrow A'$ verifying that $Ext_Q^n(A',Q)=A'$ for some $n$ and $Ext_Q^i(A',Q)=0$ for all $i\neq n$, such that $pd_Q(M\otimes_AA')<\infty$. We say that a local homomorphism of noetherian local rings $A \rightarrow B$ has finite upper Gorenstein dimension if for some Cohen factorization $A\rightarrow R \rightarrow \hat{B}$, the $R$-module $\hat{B}$ has finite upper Gorenstein dimension.

\begin{prop}\label{Gorenstein}
If $A \rightarrow B$ has the $h_2$-vanishing property and there exists a local homomorphism of noetherian local rings $g:B \rightarrow C$ such that $gf$ has finite upper Gorenstein dimension, then A is a Gorenstein ring.
\end{prop}
\begin{proof}
As in the proof of Proposition \ref{ci}, we know that there exist a Cohen factorization $A\rightarrow R \rightarrow \hat{C}$, a flat local homomorphism $R \rightarrow R'$, a regular local ring $Q$ and a surjective homomorphism $Q \rightarrow R'$ such that $Ext_Q^n(R',Q)=R'$ for some $n$ and $Ext_Q^i(R',Q)=0$ for all $i\neq n$. Now if $E'$ is the residue field of $R'$, from the change of rings spectral sequence
$$E_2^{pq}=Ext_{R'}^p(E',Ext_Q^q(R',Q)) \Rightarrow Ext_Q^{p+q}(E',Q)$$
we deduce that $R'$ is Gorenstein. By flat descent, $A$ is also Gorenstein.
\end{proof}

\begin{rem}
For a noetherian local ring $S$ let $cmd(S) := dim(S) - depth(S)$ \cite [0$_{IV}$,16.4.9]{EGA}, \cite {AF}. If we say that a finite $A$-module $M\neq 0$ has finite Cohen-Macaulay dimension when there exists a flat local homomorphism of noetherian local rings $A \rightarrow A'$ with regular closed fiber and a surjective homomorphism of noetherian local rings $Q \rightarrow A'$ verifying that $cmd(Q) = cmd(A')$ (compare with \cite [(3.2)]{AF}) and $pd_Q(M\otimes_AA')<\infty$, then we have a similar criterion for Cohen-Macaulay rings.
\end{rem}

\begin{rems}\label{h4}
(i) A detailed inspection of the proofs (using also \cite [17.13, 6.27]{An}), shows that these criteria for complete intersection, Gorenstein and Cohen-Macaulay are also valid if instead of assuming the $h_2$-vanishing property we assume the $h_4$-vanishing property (see Remark \ref{hn}.(i)). For instance, any local homomorphism factorizing through a complete intersection ring is $h_4$-vanishing but not necessarily $h_2$-vanishing. \\*
(ii) Particular cases of these facts are then:\\*
- Let $f:A \rightarrow B$ be a local homomorphism factorizing into local homomorphisms $A \rightarrow R \rightarrow B$ where $R$ is a regular local ring. If fd$_A(B)<\infty$ then $A$ is a regular ring.\\*
- Let $f:A \rightarrow B$ be a local homomorphism factorizing into local homomorphisms $A \rightarrow R \rightarrow B$ where $R$ is a complete intersection local ring. If CI*-dim$(f)<\infty$ then $A$ is a complete intersection ring.

We do not know if a similar result exists for the Gorenstein property.
\end{rems}


\begin{thebibliography}{10}

\bibitem{An}
M. Andr\'e, \emph{Homologie des Alg\`ebres Commutatives}, Springer, 1974.

\bibitem{An-b}
M. Andr\'e, \emph{Modules des diff\'erentielles en caract\'eristique p.}, Manuscripta Math. 62 (1988), no. 4, 477-502.

\bibitem{AB}
M. Auslander, M. Bridger. \emph{Stable module theory}, Memoirs of the American Mathematical Society, No. 94, 1969.

\bibitem{Av-d}
L. L. Avramov, \emph{Flat morphisms of complete intersections}, Dokl. Akad. Nauk SSSR 225 (1975), no. 1, 11-14 (English translation: Soviet Math. Dokl. 16 (1975), no. 6, 1413-1417 (1976)).

\bibitem{Av-a}
L. L. Avramov, \emph{Descente des d\'eviations par homomorphismes locaux et g\'en\'eration des id\'eaux de dimension projective finie}, C. R. Acad. Sci. Paris S\'er. I Math. 295 (1982), no. 12, 665-668.

\bibitem{Av-b}
L. L. Avramov, \emph{Locally complete intersection homomorphisms and a conjecture of Quillen on the vanishing of cotangent homology}, Ann. of Math. (2) 150 (1999), no. 2, 455-487.

\bibitem{Av-c}
L. L. Avramov, \emph{Homological dimensions and related invariants of modules over local rings}, Representations of algebra. Vol. I, II, 1-39, Beijing Norm. Univ. Press, Beijing, 2002.

\bibitem{AF}
L. L. Avramov, H.-B. Foxby, \emph{Cohen-Macaulay properties of ring homomorphisms}, Adv. Math. 133 (1998), no. 1, 54-95.

\bibitem{AFH}
L. L. Avramov, H.-B. Foxby, B. Herzog, \emph{Structure of local homomorphisms}, J. Algebra 164 (1994), no. 1, 124-145.

\bibitem{AGP}
L. L. Avramov, V. N. Gasharov, I. V. Peeva, \emph{Complete intersection dimension}, Inst. Hautes \'Etudes Sci. Publ. Math. No. 86 (1997), 67-114 (1998).

\bibitem{AHIY}
L. L. Avramov, M. Hochster, S. Iyengar, Y. Yao, \emph{Homological invariants of modules over contracting endomorphisms}, Math. Ann. 353 (2012), no. 2, 275-291.

\bibitem{AIM}
L. L. Avramov, S. Iyengar, C. Miller, \emph{Homology over local homomorphisms}, Amer. J. Math. 128 (2006), no. 1, 23-90.

\bibitem{BM}
A. Blanco, J. Majadas, \emph{Sur les morphismes d'intersection compl\`ete en caract\'eristique p}, J. Algebra 208 (1998), no. 1, 35-42.

\bibitem{EGA}
A. Grothendieck, \emph{\'El\'ements de g\'eom\'etrie alg\'ebrique. IV. \'Etude locale des sch\'emas et des morphismes de sch\'emas. I.}, Inst. Hautes \'Etudes Sci. Publ. Math. No. 20 (1964).

\bibitem{He}
J. Herzog, \emph{Ringe der Charakteristik p und Frobeniusfunktoren}, Math. Z. 140 (1974), 67-78.

\bibitem{IS}
S. Iyengar, S. Sather-Wagstaff, \emph{G-dimension over local homomorphisms. Applications to the Frobenius endomorphism}, Illinois J. Math. 48 (2004), no. 1, 241-272.

\bibitem{KL}
J. Koh, K. Lee, \emph{Some restrictions on the maps in minimal resolutions},  J. Algebra 202 (1998), no. 2, 671-689.

\bibitem{Ku}
E. Kunz, \emph{Characterizations of regular local rings of characteristic p}, Amer. J. Math.  91 (1969), 772-784.

\bibitem{Ma}
J. Majadas, \emph{A descent theorem for formal smoothness}, arXiv:1209.5055v1.

\bibitem{MR}
J. Majadas, A. G. Rodicio, \emph{Smoothness, Regularity and Complete Intersection}, London Mathematical Society Lecture Note Series, 373. Cambridge University Press, Cambridge, 2010.

\bibitem{Ra}
H. Rahmati, \emph{Contracting endomorphisms and Gorenstein modules}, Arch. Math. (Basel) 92 (2009), no. 1, 26-34.

\bibitem{Ro}
A. G. Rodicio, \emph{On a result of Avramov}, Manuscripta Math. 62 (1988), no. 2, 181-185.

\bibitem{Se}
J.-P. Serre, \emph{Sur la dimension homologique des anneaux et des modules noeth\'eriens}, Proceedings of the international symposium on algebraic number theory, Tokyo \& Nikko, 1955, pp. 175-189. Science Council of Japan, Tokyo, 1956.

\bibitem{T}
R. Takahashi, \emph{Upper complete intersection dimension relative to a local homomorphism}, Tokyo J. Math. 27 (2004), no. 1, 209-219.

\bibitem{TY}
R. Takahashi, Y. Yoshino, \emph{Characterizing Cohen-Macaulay local rings by Frobenius maps}, Proc. Amer. Math. Soc. 132 (2004), no. 11, 3177-3187.

\bibitem{Ve}
O. Veliche, \emph{Construction of modules with finite homological dimensions}, J. Algebra 250 (2002), no. 2, 427-449.



\end{thebibliography}

\end{document}